\documentclass[11pt,reqno]{amsart}

\usepackage{amsmath}
\usepackage{amssymb}
\usepackage{amsthm}
\usepackage{enumerate}
\usepackage[mathscr]{eucal}

\usepackage{mathtools}
\usepackage{microtype}

\usepackage[english]{babel}

\usepackage{amsmath,amssymb,latexsym,textcomp,mathrsfs}
\usepackage[all]{xy}
\usepackage{graphicx}
\usepackage{bm,amsmath, amsthm, amssymb, amsfonts}
\usepackage{times}
\usepackage[english]{babel}

\setbox0=\hbox{$+$}
\newdimen\plusheight
\plusheight=\ht0
\def\+{\;\lower\plusheight\hbox{$+$}\;}  

\setbox0=\hbox{$-$}
\newdimen\minusheight
\minusheight=\ht0
\def\-{\;\lower\minusheight\hbox{$-$}\;}

\setbox0=\hbox{$\cdots$}
\newdimen\cdotsheight
\cdotsheight=\plusheight
\def\cds{\lower\cdotsheight\hbox{$\cdots$}}

\numberwithin{equation}{section}
\newtheorem{thm}{Theorem}[section]
\newtheorem{lem}[thm]{Lemma}

\newtheorem{prop}[thm]{Proposition}

\newtheorem{cor}{Corollary}

\theoremstyle{definition}
\newtheorem{defn}{Definition}[section]

\newtheorem{exmp}{Example}[section]

\theoremstyle{remark}
\newtheorem{rem}{\bf{Remark}}
\newtheorem{note}{\bf{Note}}
\numberwithin{equation}{section}

\fussy

\begin{document}

\setcounter {page}{1}
\title{ On rough $I^*$ and $I^K$-convergence of sequences in  normed linear spaces}

\author{Amar Kumar Banerjee and Anirban Paul}

\address[Amar Kumar Banerjee]{Department of Mathematics, The University of Burdwan, Golapbag, Burdwan-713104, West Bengal, India.}

\email{akbanerjee@math.buruniv.ac.in, akbanerjee1971@gmail.com }

\address[Anirban Paul]{Department of Mathematics, The University of Burdwan, Golapbag, Burdwan-713104, West Bengal, India.}
         
\email{paulanirban310@gmail.com}

\begin{abstract} 
In this paper, we have introduced first the notion of rough $I^*$-convergence in a normed linear space as an extension work of rough $I$-convergence and then rough $I^K$-convergence in more general way. Then we have studied some properties on these two newly introduced ideas. We also examined the relationship between rough $I$-convergence with both of rough $I^*$-convergence and rough $I^K$-convergence. 
\end{abstract}
\maketitle
\author{}


\textbf {(2010) AMS subject classification :} 54A20, 40A35. \\
\section{Introduction}
The concepts of convergence of a sequence of real numbers has been generalised to statistical convergence independently by Fast \cite{f} and Steinhaus \cite{q}. Then a lot of developments were carried out in this area by many authors. The concepts of statistical convergence of sequence has been extended to $I$-convergence by Kostyrko et al. \cite{h,t**} using the structure of the ideal $I$ of subsets of the set of natural numbers. An another type of convergence which is closely related to the ideas of $I$-convergence is the idea of $I^*$-convergence given by Kostyrko et al. \cite{4th4}. It is seen in \cite{h} that these notions are equivalent if and only if the ideal satisfies the property (AP). Several works have been done in recent years on $I$-convergence (see \cite{4,3,5,6,i,j}).\\
The idea of rough convergence in a finite dimensional space was introduced by Phu \cite{R2} in 2001. In 2014, D\"undar et al. \cite{R7} introduced the notion rough $I$-convergence using the concepts of $I$-convergence and rough  convergence. For given two arbitrary ideals $I$ and $K$ on a set $S$, the idea of $I^K$-convergence in topological space was given by M. Ma\v{c}aj and M. Sleziak in \cite{4th2} as a generalization of the notion of $I^*$-convergence. In their paper they modified the condition (AP) and have showed that when such condition termed as AP($I$, $K$) holds then $I$-convergence implies $I^K$-convergence and the converse of this result also holds for the first countable space which is not finitely generated( or Alexandroff space). Indeed, they worked with functions instead of sequences. One of the reasons is that using functions sometimes helps to simplify notations.\\
In our work we have introduced the notion of rough $I^*$-convergence and rough $I^K$-convergence in a normed linear space. Rough $I^K$-convergence is a common generalization of rough $I^*$-convergence. Here we have studied the ideas of rough $I^*$-convergence and rough $I^K$-convergence in terms of sequences instead of functions. We then intend to find the relation between rough $I$-convergence and rough $I^*$-convergence. We also tried to find the relation of rough $I$-convergence with rough $I^K$-convergence and we have observed that the condition AP($I$,$K$) is an necessary and sufficient condition for the rough $I$-limit set to be a subset of rough $I^K$-limit set. We have tried to verify whether some property of $I^K$-convergence as in \cite{4th2} also holds for rough $I^K$-convergence.
\\
We now recall some definitions and notions which will be needed in sequel.
\section{Preliminaries}
Throughout the paper, $\mathbb{N}$ denotes the set of all natural numbers, $\mathbb{R}$ the set of all real numbers unless otherwise stated.
\begin{defn}\cite{f}
Let $K$ be a subset of the set of natural numbers $\mathbb{N}$ and let us denote the set $K_i =\{k\in K : k\leq i\}$. Then the natural density of $K$ is given by $d(K)=\displaystyle{\lim_{i\rightarrow \infty}}\frac{|K_i|}{i}$, where $|K_i|$ denotes the number of elements in $K_i$. 
\end{defn}
\begin{defn}\cite{f}
A sequence $\{x_n\}_{n\in\mathbb{N}}$ of real numbers is said to be statistically convergent to $x$ if for any $\varepsilon >0$, $d(A(\varepsilon))=0$, where $A(\varepsilon)=\{n\in \mathbb{N}: |x_n - x|\geq \varepsilon\}$.
\end{defn}
Let $I$ be a collection of subset of a set $S$. Then $I$ is called an ideal on $S$ if $(i)$ $A, B\in I$ $\Rightarrow A\cup B\in I$ and $(ii)$ $A\in I$ and $B\subset A$ $\Rightarrow B\in I$ \cite{t}.\\
An ideal $I$ on $S$ is called admissible if it contains all singletons, that is, $\{s\}\in I$ for each $s\in S$. $I$ is called nontrivial if $S\notin I$ \cite{t}. From the definition it is noted that $\phi\in I$.\\
If $S=\mathbb{N}$, the set of all positive integers then $I$ is called an ideal on $\mathbb{N}$. We will denote by Fin the ideal of all finite subsets of a given set $S$. 
\begin{defn}\cite{4th3}
Let $S\neq \phi$. A non empty class $F$ is called a filer in $S$ provided: $(i)$ $\phi\notin F$, $(ii)$ $A, B\in F\Rightarrow A\cap B\in F$, $(iii)$ $A\in F, A\subset B \Rightarrow B\in F$.
\end{defn}
\begin{lem}\cite{h}
If $I$ is a non trivial ideal on $\mathbb{N}$, then the class $F(I)=\{M\subset \mathbb{N}: \text{there exists } A\in I\:\: \text{such that }\: M=\mathbb{N}\setminus A\}$ is a filter on $\mathbb{N}$, called the filter associated with $I$.
\end{lem}
\begin{defn}\cite{h}
An admissible ideal $I\subset 2^\mathbb{N}$ is said to satisfy the condition (AP) if for every countable family of mutual disjoint sets $\{A_1, A_2,\cdots\}$ belonging to $I$ there exists a countable family of sets $\{B_1, B_2,\cdots\}$ such that the symmetric difference $A_j\Delta B_j$ is a finite set for each $j\in\mathbb{N}$ and $B=\displaystyle{\bigcup_{j=1} ^ \infty} B_j \in I$. Several example of countable family satisfying (AP) are seen in \cite{h}.
\end{defn}
\begin{defn}\cite{h,t**}
Let $(X, ||\cdot||)$ be a normed linear space and $I\subset 2^\mathbb{N}$ be a non-trivial ideal. A sequence $\{x_n\}_{n\in\mathbb{N}}$ of elements of $X$ is said to be $I$-convergent to $x\in X$ if for each $\varepsilon >0$ the set $A(\varepsilon)=\{n\in\mathbb{N}: ||x_n - x||\geq \varepsilon\}$ belongs to $I$. The element $x$ is here called the $I$-limit of the sequence $\{x_n\}_{n\in\mathbb{N}}$.
\end{defn}
It should be noted here that if $I$ is an admissible ideal then usual convergence in $X$ implies $I$-convergence in $X$.
\begin{exmp}\cite{h}
If $I_d$ denotes the class of all $A\subset \mathbb{N}$ with $d(A)=0$. Then $I_d$ is non trivial admissible ideal and $I_d$ convergence coincides with the statistical convergence.
\end{exmp}
\begin{defn}\cite{4th4,h}
Let $(X, ||\cdot||)$ be a normed linear space and $I\subset 2^\mathbb{N}$ be a non-trivial ideal. A sequence $\{x_n\}_{n\in\mathbb{N}}$ in $X$ is said to be $I^*$-convergent to $x$ if there exists a set $M=\{m_1<m_2<\cdots<m_k<\cdots\}$ in $F(I)$ such that the sub sequence $\{x_{m_k}\}_{k\in \mathbb{N}}$ is convergent to $x$ i.e., $\displaystyle{\lim_{k\rightarrow \infty}} ||x - x_{m_k}|| =0$
\end{defn}
It is seen in \cite{h} that $I^*$-convergence implies $I$-convergence. If an admissible ideal $I$ has the property (AP), then for a sequence $\{x_n\}_{n\in\mathbb{N}}$ in a normed linear space $X$, $I$-convergence implies $I^*$-convergence.
\begin{defn}\label{correction}\cite{R2}
Let $\{x_n\}_{n\in\mathbb{N}}$ be a sequence in a normed linear space $(X, ||\cdot||)$ and $r$ be a non-negative real number. Then $\{x_n\}_{n\in\mathbb{N}}$ is said to be rough convergent of roughness degree $r$ to $x$ or simply $r$-convergent to $x$, denoted by $x_n \xrightarrow{r} x$, if for all $\varepsilon >0$ there exists $N(\varepsilon)\in\mathbb{N}$ such that $n\geq N(\varepsilon)$ implies $|| x_n - x||< r +\varepsilon$ and $x$ is called rough limit of $\{x_n\}_{n\in\mathbb{N}}$ of roughness degree $r$.
\end{defn}
For $r=0$ the definition \ref{correction} reduces to definition of classical convergence of sequences. Here $x$ is called the $r$-limit point of $\{x_n\}_{n\in\mathbb{N}}$, which is usually no more unique (for $r>0$). So we have to consider the so called $r$-limit set (or shortly $r$-limit) of $\{x_n\}_{n\in\mathbb{N}}$ defined by $LIM^r x_n :=\{ x\in X : x_n \xrightarrow{r} x\}$. A sequence $\{x_n\}_{n\in\mathbb{N}}$ is said to be $r$-convergent if $LIM^r x_n \neq \phi$. In this case, $r$ is called a rough convergence degree of $\{x_n\}_{n\in\mathbb{N}}$.
\begin{prop}\cite{R2}
A sequence $\{x_n\}_{n\in\mathbb{N}}$ in a normed linear space $X$ is bounded if and only if there exists an $r\geq 0$ such that $LIM^r x_n\neq \phi$. For all $r>0$, a bounded sequence $\{x_n\}_{n\in\mathbb{N}}$ always contains a subsequence $\{x_{m_k}\}_{k\in\mathbb{N}}$ with $LIM^r x_{m_k}\neq \phi$.
\end{prop}
\begin{prop}\cite{R2}
If $\{x'_n\}_{n\in\mathbb{N}}$ is a sub sequence of $\{x_n\}_{n\in\mathbb{N}}$ in a normed linear space $(X, ||\cdot||)$, then $LIM^r x_n \subset LIM^r x'_n$.
\end{prop}
\begin{prop}\cite{R2}
For all $r\geq 0$, the $r$-limit set $LIM^r x_n$ of an arbitrary sequence $\{x_n\}_{n\in\mathbb{N}}$ in a normed linear space $(X, ||\cdot||)$ is closed set.
\end{prop}
\begin{prop}\cite{R2}
For an arbitrary sequence $\{x_n\}_{n\in\mathbb{N}}$ in a normed linear space $(X, ||\cdot||)$ the $r$-limit set $LIM^r x_n$ is convex.
\end{prop}
\begin{defn}\cite{R7}
A sequence $\{x_n\}_{n\in\mathbb{N}}$ in a normed linear space $(X, ||\cdot||)$ is said to be $I$-bounded if there exists a positive real number $M$ such that the set $\{n\in\mathbb{N}: ||x_n||\geq M\}\in I$
\end{defn}
\begin{defn}\cite{R7}
Let $I$ be an admissible ideal and $r$ be a non-negative real number. A sequence $\{x_n\}_{n\in\mathbb{N}}$ in a normed linear space $(X, ||\cdot||)$ is said to be rough $I$-convergent of roughness degree $r$ to $x$, denoted by $x_n \xrightarrow { r-I} x$ provided that $\{n\in\mathbb{N}: ||x_n - x||\geq r +\varepsilon\}\in I$ for every $\varepsilon >0$ and $x$ is called rough $I$-limit of $\{x_n\}_{n\in\mathbb{N}}$ of roughness degree $r$.
\end{defn}
\begin{rem}
If $I$ is an admissible ideal, then the usual rough convergence implies rough $I$-convergence.
\end{rem}
\begin{note}
If we take $r=0$, then we obtain the definition of ordinary $I$-convergence. In general, the rough $I$-limit of a sequence may not be unique for the roughness degree $r>0$. So we have to consider the so-called rough $I$-limit set of a sequence $\{x_n\}_{n\in\mathbb{N}}$ which is defined by $I-LIM^r x_n :=\{x\in X: x_n\xrightarrow{r -I} x\}$. A sequence $\{x_n\}_{n\in\mathbb{N}}$ is said to be rough $I$-convergent if $I-LIM^r x_n \neq \phi$.
\end{note}
\begin{thm}\cite{R7}
Let $I\subset 2^\mathbb{N}$ be an admissible ideal. A sequence $\{x_n\}_{n\in\mathbb{N}}$ in $(X,||\cdot||)$ is $I$-bounded if and only if there exists a non negative real number $r$ such that $I-LIM^r x_n\neq \phi$.
\end{thm}
\begin{thm}\cite{R7}
Let $I\subset 2^\mathbb{N}$ be an admissible ideal. If $\{x_{m_k}\}_{k\in\mathbb{N}}$ is a sub sequence of $\{x_n\}_{n\in\mathbb{N}}$ in $(X, ||\cdot||)$ then $I-LIM^r x_n \subset I- LIM^r x_{m_k} $
\end{thm}
\begin{thm}\cite{R7}
Let $I\subset 2^\mathbb{N}$ be an admissible ideal. The rough $I$-limit set of a sequence $\{x_n\}_{n\in\mathbb{N}}$ in $(X, ||\cdot||)$ is closed.
\end{thm}
\begin{thm}\cite{R7}
Let $I\subset 2^\mathbb{N}$ be an admissible ideal. The rough $I$-limit set of a sequence $\{x_n\}_{n\in\mathbb{N}}$ in $(X, ||\cdot||)$ is convex.
\end{thm}
We now give some basic ideas on $I^K$-convergence in a topological space studied by M. Ma\v{c}aj and M. Sleziak \cite{4th2}.
\begin{defn}\cite{4th2}
Let $I$ be an ideal on a set $S$ and $X$ be a topological space. A function $f: S\mapsto X$ is said to be $I$-convergent to $x$ if $f^{-1} (U)=\{s\in S: f(s)\in U\}\in F(I)$ holds for every neighbourhood $U$ of the point $x$ and we write $I-lim f=x$
\end{defn}
If $S=\mathbb{N}$, then we obtain the usual definition of $I$-convergence of sequences. In this case the notation $I-lim x_n=x$ is used.
\begin{defn}\cite{4th2}
Let $I$ be an ideal on a set $S$ and let $f:S\mapsto X$ be a function to a topological space $X$. The function $f$ is called $I^*$-convergent to the point $x$ if there exists a set $M\in F(I)$ such that the function $g:S\mapsto X$ defined by $g(s)=\begin{cases}
f(s), \:\text{if} s\in M\\
x, \:\text{if} s\notin M
\end{cases}$ is Fin-convergent to $x$. If $f$ is $I^*$-convergent to $x$, then we write $I^*-lim f =x$.
\end{defn}
The usual notion of $I^*$-convergence of sequences is a special case when $S=\mathbb{N}$.
\begin{defn}\label{eqidef}\cite{4th2}
Let $K$ and $I$ be ideals on a set $S$ and $X$ be a topological space. The function $f:S\mapsto X$ is said to be $I^K$-convergent to $x\in X$ if there exists a set $M\in F(I)$ such that the function $g:S \mapsto X$ given by $g(s)=\begin{cases}
f(s), \text{if} \:\:s\in M\\
x, \text{if} \:\:s\notin M
\end{cases}$ is $K$-convergent to $x$. If $f$ is $I^K$-convergent to $x$, then we write $I^K-lim f =x$.
\end{defn}
When $S=\mathbb{N}$, then we speak about $I^K$-convergence of sequences.
\begin{rem}
The definition of $I^K$-convergence may also be treated from \cite{4th4} as follows: there exists $M\in F(I)$ such that the function $f|_M$ is $K|M$-convergent to $x$, where $K|M=\{A\cap M: A\in K\}$ is the trace of $K$ on $M$. The two definitions are equivalent but the definition given in \ref{eqidef} is somewhat simpler \cite{4th2}.
\end{rem}
\begin{lem}\cite{4th2}
If $I$ and $K$ are ideals on a set $S$ and $f: S\mapsto X$ is a function such that $K-lim f=x$ then $I^K-lim f=x$.
\end{lem}
\begin{defn}\cite{4th2}
Let $K$ be an ideal on a set $S$. we write $A\subset_K B$ whenever $A\setminus B\in K$. if $A\subset_K B$ and $B\subset_K A$ then we write $A \sim_K B$.
\end{defn}
Clearly $A\sim_K B \Leftrightarrow A\Delta B\in K$.\\
Now we recall a lemma from \cite{4th2} where several equivalent formulations of a condition for ideals $I$ and $K$ have been described.
\begin{lem}\label{apk}\cite{4th2}
Let $I$ and $K$ be ideals on the same set $S$. Then the following conditions are equivalent:\\
$(i)$ For every sequence $\{A_n\}_{n\in\mathbb{N}}$ of sets from $I$ there is $A\in I$ such that $A_n\sim _K A$ for all $n$'s.\\
$(ii)$ Any sequence $\{F_n\}_{n\in\mathbb{N}}$ of sets from $F(I)$ has a $K$-pseudo intersection in $F(I)$.\\
$(iii)$ For every sequence $\{A_n\}_{n\in\mathbb{N}}$ of sets belonging to $I$ there exists a sequence $\{B_n\}_{n\in\mathbb{N}}$ of sets from $I$ such that $A_j \sim_K B_j$ for $j\in \mathbb{N}$ and $B=\bigcup _{j\in\mathbb{N}} B_j \in I$.\\
$(iv)$ For every sequence of mutually disjoint sets $\{A_n\}_{n\in\mathbb{N}}$ belonging to $I$ there exists a sequence $\{B_n\}_{n\in\mathbb{N}}$ of sets belonging to $I$ such that $A_j\sim_K B_j$ for $j\in\mathbb{N}$ and $B=\bigcup _{j\in\mathbb{N}} B_j \in I$.\\
$(v)$ For every non-decreasing sequence $A_1\subset A_2 \subset \cdots \subset A_n\subset \cdots$ of sets from $I$ there exists a sequence $\{B_n\}_{n\in\mathbb{N}}$ of sets belonging to $I$ such that $A_j \sim_K B_j$ for $j\in\mathbb{N}$ and $B=\bigcup_{j\in\mathbb{N}} B_j\in I$.

\end{lem}
\begin{defn}\cite{4th2}
Let $I$ and $K$ be two ideals on a same set $S$. We say that $I$ has the additive property with respect to $K$, more briefly that AP$(I, K)$ holds, if any one of the equivalent conditions of Lemma \ref{apk} holds.
\end{defn}
The condition (AP) from \cite{h}, is equivalent to the condition AP($I$, Fin). Now we recall the following two theorems from \cite{4th2}.
\begin{thm}\cite{4th2}
Let $I$ and $K$ be ideals on a set $S$ and $X$ be a first countable topological space. If $I$ has the additive property with respect to $K$, then for any function $f: S \mapsto X$ $I$ convergence implies the $I^K$-convergence or in other words, if the condition AP$(I, K)$ holds then the $I$-convergence implies the $I^K$-convergence.
\end{thm}
Let us recall that a topological space $X$ is called finitely generated space or Alexandroff space if intersection of any number of open sets of $X$ is again an open set (see \cite{4th1}). Equivalently, $X$ is finitely generated if and only if each point $x$ has a smallest neighbourhood.
\begin{thm}\cite{4th2}
Let $I$, $K$ be ideals on a set $S$ and $X$ be a first countable topological space which is not finitely generated. If the $I$-convergence implies the $I^K$-convergence for any function $f: S\mapsto X$, then the ideal $I$ has the additive property with respect to $K$ or briefly the condition AP$(I,K)$ holds.
\end{thm}

\section{Main results}
Through out our discussion  $(X, ||\cdot||)$ or simply $X$ will always denote a normed linear space over the field $\mathbb{C}$ or $\mathbb{R}$ and $I$,$K$ always assumed to be non trivial admissible ideals on $\mathbb{N}$ unless otherwise stated.
\begin{defn}\label{def1}
Let $r$ be a non-negative real number and $I$ be a non trivial admissible ideal on $\mathbb{N}$. Then a sequence $\{x_n\}_{n\in \mathbb{N}}$ in $(X, || \cdot || )$ is said to be rough $I^*$-convergent of roughness degree $r$ to $x$ if there exists a set $M = \{m_1 < m_2< m_3 < \cdots < m_k < \cdots\}$ in $F(I)$ such that the sub sequence $\{x_{m_k}\}_{k\in \mathbb{N}}$ is rough convergent of roughness degree $r$ to $x$. Thus for any $\varepsilon >0$ there exists a $N\in \mathbb{N}$ such that $|| x_{m_k} - x || < r + \varepsilon $ for all $k \geq N$. we denote this by $x_n \xrightarrow{r-I^*} x$.
\end{defn}
Here $x$ is called the rough $I^*$-limit of the sequence $\{x_n\}_{n\in \mathbb{N}}$ of roughness degree $r$. For $r=0$ we have the definition of $I^*$-convergence of sequences in normed linear spaces. Obviously rough $I^*$-limit of a sequence in normed linear spaces is not unique. Therefore we have to consider the rough $I^*$-limit set of the sequence $\{x_n\}_{n\in\mathbb{N}}$ defined as follows: $I^*-LIM^r x_n = \{x\in X :x_n \xrightarrow{r - I^*} x\}$. 
\begin{defn}\label{def2}
Let $r$ be a non-negative real number. Also let $I$ and $K$ be two non trivial admissible ideals on $\mathbb{N}$. Then a sequence $\{x_n\}_{n\in \mathbb{N}}$ in a normed linear space $(X, || \cdot || )$ is said to be rough $I^K$-convergent of roughness degree $r$ to $x$ if there exists a set $M=\{m_1< m_2<\cdots <m_k<\cdots\}$ in $F(I)$ such that the sub sequence $\{x_{m_k}\}_{k\in \mathbb{N}}$ is rough $K|M$-convergent of roughness degree $r$ to $x$, where $K|M=\{A\cap M: A\in K\}$ is the trace of $K$ on $M$. That is for any $\varepsilon >0$, the set $\{k\in \mathbb{N} : || x_{m_k} - x || \geq r +\varepsilon\}\in K|M$. we denote this by $x_n \xrightarrow{r-I^K} x$.
\end{defn}
Here $x$ is called the rough $I^K$-limit of the sequence $\{x_n\}_{n\in\mathbb{N}}$ of roughness degree $r$. For $r=0$ we have the definition of $I^K$-convergent of sequences in normed linear spaces. It should be noted that for $M\in F(I)$, the trace $K|M=\{A\cap M: A\in K\}$ of $K$ on $M$ also forms an ideal on $\mathbb{N}$. Clearly rough $I^K$-limit of a sequence in normed linear spaces is not unique. Therefore we will consider the rough $I^K$-limit set of the sequence $\{x_n\}_{n\in \mathbb{N}}$ defined by $I^K-LIM^r x_n =\{x\in X : x_n \xrightarrow{r-I^K} x\}$. If the ideal $K$ is such that it is the class of all finite subsets of $\mathbb{N}$ then definition \ref{def1} and definition \ref{def2} coincides. Obviously if $x$ is a rough $I^*$-limit of a sequence $\{x_n\}_{n\in\mathbb{N}}$ then $x$ is also a rough $I^K$-limit of $\{x_n\}_{n\in\mathbb{N}}$. But it may happen that $x$ is rough $I^K$-limit of a sequence $\{x_n\}_{n\in\mathbb{N}}$ in normed linear space without being rough $I^*$-limit of the the sequence $\{x_n\}$, which is seen from the next example. So, in general, for a sequence $\{x_n\}_{n\in\mathbb{N}}$ in a normed linear space and for any non-negative real number $r$, we have $I^* -LIM^r x_n \subset I^K-LIM^r x_n$. 
\begin{exmp}
 Let us consider a decomposition of $\mathbb{N}$ by $\mathbb{N}= A\cup\displaystyle{\bigcup_{i=1} ^\infty} A_i$, where $A=\{1,3,5,\cdots\}$ and $A_i= \{2^n (2i -1): n\in \mathbb{N}\}$. Then each of $A_i$'s are disjoint from each other and each of $A_i$'s are disjoint from $A$ also. Let $I$ be the collections of all those subsets of $\mathbb{N}$ such that the sets which belongs to $I$ can intersects with $A$ and  with only a finite numbers of $A_i$'s. Then $I$ is an non trivial admissible ideal on $\mathbb{N}$. Let $\mathbb{N}= \displaystyle{\bigcup_{j=1} ^ \infty} D_j$ be another decomposition of $\mathbb{N}$ such that $D_j=\{2^{j-1}(2s -1): s=1, 2, \cdots\}$. Then each of $D_j$ is infinite and $D_j \cap D_k =\phi$ for $j\neq k$. Let $K$ be the ideal of all those subsets of $\mathbb{N}$ which intersects with only a finite numbers of $D_j$'s. Then $K$ is a non trivial admissible ideal on $\mathbb{N}$. Let us consider the sequence in real number space with usual norm define by $x_n =\frac{1}{j}$ if $n\in D_j$. Let us take $M=\mathbb{N}\in F(I)$. Then $K|M=K$. Now let $r >0$ be arbitrary. Since by Archimedean property for any arbitrary $\varepsilon >0$ there exists a $l\in \mathbb{N}$ such that $\varepsilon>\frac{1}{l}$. So $\{k\in \mathbb{N} : | x_k - (-r)|=|x_k + r|\geq r + \varepsilon\}\subset D_1\cup D_2\cup\cdots\cup D_l\in K=K|M$. Therefore $-r\in I^K-LIM^r x_n$.\\
If possible let $- r\in I^*-LIM^r x_n$. So there exists a set $M=\{m_1<m_2<\cdots <m_k<\cdots\}\in F(I)$ for which the sub sequence $\{x_n\}_{n\in M}$ of the sequence $\{x_n\}_{n\in\mathbb{N}}$ is rough convergent to $x$ of roughness degree $r$. Now as $M\in F(I)$ therefore we have $\mathbb{N}\setminus M = H \:\text{(say)}\:\in I$. Therefore exists a $p\in \mathbb{N}$ such that $H\subset A\cup A_1 \cup A_2 \cup\cdots\cup A_p$ and so $A_{k}\subset M$ for all $k\geq {p+1}$. Now as each of the set $A_k$'s contains an element from each of the set $D_i$'s for $i\geq 2$, so there exists a $s\in\mathbb{N}$ such that $x_{m_k}=\frac{1}{s}$ for infinitely $k$'s when $m_k\in D_s$. As $- r \in I^*-LIM^r x_n$, so for  $\varepsilon=\frac{1}{s+1}$ there exists a $N\in\mathbb{N}$ such that $|x_{m_k} - (- r)|=|x_{m_k} + r|< r +\varepsilon$ for all $k\geq N \rightarrow (i)$. Since $x_{m_k}=\frac{1}{s}$ for infinitely many $k$'s, therefore the condition in $(i)$ does not holds. Thus we arrived at a contradiction. Hence $- r\notin I^*-LIM^r x_n$.
\end{exmp}
\begin{thm}\label{th1}
Let $\{x_n\}_{n\in\mathbb{N}}$ be a sequence in a normed linear space $(X, ||\cdot||)$ and $r$ be a non-negative real number. Then for an non trivial admissible ideal $I$ if $\{x_n\}_{n\in \mathbb{N}}$ is rough $I^*$-convergent of roughness degree $r$ to $x$ then it is also rough $I$-convergent of roughness degree $r$ to $x$.
\end{thm}
\begin{proof}
If possible let $\{x_n\}_{n\in\mathbb{N}}$ be rough $I^*$-convergent of roughness $r$ to $x$. Therefore there exists a set $M=\{m_1<m_2<\cdots<m_k<\cdots\}\in F(I)$ such that $\{x_{m_k}\}_{k\in\mathbb{N}}$ is rough convergent of roughness degree $r$ to $x$. Thus for any $\varepsilon >0$ there exists $N\in \mathbb{N}$ such that $||x_{m_k} - x||< r +\varepsilon $ for all $ k\geq N$. So $\{k\in\mathbb{N} : ||x_k - x|| \geq r +\varepsilon \} \subset {\mathbb{N}\setminus M} \cup \{m_1, m_2,\cdots, m_{N - 1}\}\rightarrow(i)$. Now since the right hand side of $(i)$ belongs to $I$, so $\{k\in\mathbb{N} : ||x_k - x|| \geq r +\varepsilon \} \in I$. Therefore the sequence $\{x_n\}_{n\in\mathbb{N}}$ is rough $I$-convergent of roughness degree $r$ to $x$.
\end{proof}
In view of Theorem \ref{th1}, it follows that rough $I*$-limit set of roughness degree $r$ is a subset of rough $I$-limit set of same roughness degree $r$.
Converse of the theorem \ref{th1} not necessarily true. That is if a sequence $\{x_n\}_{n\in\mathbb{N}}$ is rough $I$-convergent of some roughness degree $r$ to $x$ then the sequence $\{x_n\}_{n\in\mathbb{N}}$ may not be rough $I^*$-convergent of same roughness degree $r$ to $x$. This fact can be seen from the next example.
\begin{exmp}\label{3.2}
Let $\mathbb{N}=\displaystyle{\bigcup_ {j=1 }^ \infty} D_j$ be a decomposition of $\mathbb{N}$ such that $D_j =\{2^{(j -1)} (2s -1): s=1, 2, \cdots\}$. Then each of $D_j$ is infinite and disjoint from each others. Let $I$ be the class of all those subsets of $\mathbb{N}$ which intersects with only a finite numbers of $D_j$'s. Then $I$ is an admissible ideal on $\mathbb{N}$. Let us define a sequence in real numbers space with usual norm by $x_n=\frac{1}{j^j}$ if $n\in D_j$. Let $r$ be an arbitrary non-negative real number. Let $\varepsilon >0$ be arbitrarily chosen, then there exists a $l\in \mathbb{N}$ such that $\varepsilon >\frac{1}{l^l}$. Then $[-r, r]\subset I-LIM^r x_n$, as $\{n\in\mathbb{N} : |x_n - x |\geq r + \varepsilon\}\subset  D_1 \cup D_2 \cup \cdots D_l\in I$ for any $x\in [-r, r]$.\\
If possible suppose that the sequence defined above is rough $I^*$-convergent to $ - r$ of same roughness degree $r$. Therefore, there exists a set $M=\{m_1< m_2<\cdots< m_k <\cdots\}\in F(I)$ such that the sub sequence $\{x_{m_k}\}_{k\in \mathbb{N}}$ is rough convergent to $-r$ of roughness degree $r$. Now as $M\in F(I)$, so $\mathbb{N}\setminus M= H(\text{say})\in I$. Hence there exists a $p\in \mathbb{N}$ such that $H\subset D_1 \cup D_2 \cup\cdots\cup D_p$ and so $D_{p + 1 } \subset M$. Therefore $x_{m_k}= \frac{1}{(p +1) ^ {p+1}}$ for $m_k\in D_{p+1}$. Now for $\varepsilon =\frac{1}{(p+2)^{p+1}}$ and $m_k\in D_{p+1}$ we see that $|x_{m_k} + r|\geq r +\varepsilon$ for infinitely many $k$'s. Therefore the sequence $\{x_n\}_{n\in\mathbb{N}}$ is not rough $I^*$-convergent of roughness degree $r$ to $-r$ although $- r\in I-LIM^r x_n$.
\end{exmp} 
Let $r$ be a non-negative real number. Then for a sequence $\{x_n\}_{n\in \mathbb{N}}$ in a normed linear space rough $I$-limit of the sequence $\{x_n\}_{n\in\mathbb{N}}$ of roughness degree $r$ is also a rough $I^*$-limit of same roughness degree $r$ if the ideal $I$ satisfies the condition (AP). To prove this we need the following lemma.
\begin{lem}\cite{m}\label{lem1}
Let $\{A_n\}_{n\in\mathbb{N}}$ be a countable family of subsets of $\mathbb{N}$ such that each $A_n$ belongs to  $F(I)$, the filter associated with an admissible ideal $I$ which has the property (AP). Then there exists a set $B\subset \mathbb{N}$ such that $B\in F(I)$ and the set $B\setminus A_n$ is finite for all $n\in \mathbb{N}$.
\end{lem}
\begin{thm}\label{th2}
Let $I$ be an ideal which has the property (AP) and $\{x_n\}_{n\in \mathbb{N}}$ be a sequence in a normed linear space $(X, ||\cdot||)$. Then if $x$ is a rough $I$-limit of the sequence $\{x_n\}_{n\in\mathbb{N}}$ of some roughness degree $r$ then $x$ is also a rough $I^*$-limit of the sequence $\{x_n\}_{n\in\mathbb{N}}$ of same roughness degree $r$.
\end{thm}
\begin{proof}
Let $I$ be an ideal on $\mathbb{N}$ which satisfies the condition (AP) and $\{x_n\}_{n\in\mathbb{N}}$ be a sequence in a normed linear space $(X, ||\cdot|| )$. Also let us suppose that $x$ be a rough $I$-limit of the sequence $\{x_n\}_{n\in\mathbb{N}}$ of roughness degree $r$ for some $r\geq 0$. Therefore for any $\varepsilon >0$ the set $\{n\in \mathbb{N} : ||x_n - x||\geq r +\varepsilon\}\in I$. Let $l$ be any arbitrary positive real number, so $\frac{l}{i} $ is also positive real number for each $i\in \mathbb{N}$. Define $A_i =\{n\in\mathbb{N} : || x_n - x || < r + \frac{l}{i}\}$ for each $i\in \mathbb{N}$. Then $A_i\in F(I)$ for each $i\in \mathbb{N}$. Also by the lemma \ref{lem1} there exists a set $B\subset \mathbb{N}$ such that $B\in F(I)$ and $B\setminus A_i$ is finite for all $i\in \mathbb{N}$. Now for any arbitrary $\varepsilon >0$ there exists a $j\in \mathbb{N}$ such that $\varepsilon > \frac{l}{j}$. Since $B\setminus A_j$ is finite, so there exists $k=k(j)\in \mathbb{N}$ such that $n\in B\cap A_j$ for all $n\in B$ with $n\geq k$. Now $|| x_n - x|| < r + \frac{l}{j}< r + \varepsilon$ for all $n\in B$ and $n\geq k$. Thus the sub sequence $\{x_n\}_{n\in B}$ is rough convergent of roughness degree $r$ to $x$. Therefore $x$ is also a rough $I^*$-limit of roughness degree $r$. Hence the result follows.
\end{proof}
\begin{cor}
Let $\{x_n\}_{n\in\mathbb{N}}$ be a sequence in a normed linear space $(X, || \cdot||)$ and $r$ be a non-negative real number. Let $I$ be an ideal on $\mathbb{N}$ such that it satisfies the condition (AP). Then both the rough $I$-limit set of roughness degree $r$ and rough $I^*$-limit set of roughness degree $r$ of a sequence $\{x_n\}_{n\in\mathbb{N}}$ are equal.
\end{cor}
\begin{proof}
In view of theorem \ref{th1} and theorem \ref{th2} the result follows.
\end{proof}
Rough $I$-limit set of a sequence $\{x_n\}_{n\in\mathbb{N}}$ in a normed linear space is a subset of rough $I$-limit set of a sub sequence $\{x_{n_k}\}_{k\in\mathbb{N}}$. But rough $I^*$-limit set of a sequence $\{x_n\}_{n\in\mathbb{N}}$ in a normed linear space may not be a subset of rough $I^*$-limit set of a sub sequence $\{x_{n_k}\}_{k\in\mathbb{N}}$. This fact can be justified by the following example.

\begin{exmp}
Let $I$ be the ideal of all subsets of $\mathbb{N}$ whose natural density is zero. Let us consider a sequence $\{x_n\}_{n\in\mathbb{N}}$ in real number space with usual norm as follows: $x_n=
\begin{cases}
 -1 & n=k^2\\
\frac{1}{n} & n\neq k^2
\end{cases}$, where $k\in\mathbb{N}$. Now as the natural density of the set $A=\{n\in\mathbb{N}: n=k^2,\: k\in\mathbb{N}\}$ is zero, therefore $A\in I$. So $\mathbb{N}\setminus A=M(\text{say})\in F(I)$. Put $M=\{m_1<m_2<\cdots<m_k<\cdots\}$. Then for any arbitrary $\varepsilon>0$ we can see that $|x_{m_k} - 1| < 1 + \varepsilon$ holds for all $k\in\mathbb{N}$. Hence $1$ is a rough $I^*$-limit of roughness degree $r=1$ of the sequence $\{x_n\}_{n\in\mathbb{N}}$. Let $A$ be enumerated as $A=\{n_1<n_2<\cdots<n_k<\cdots\}$ and consider the sub sequence $\{x_{n_k}\}_{k\in\mathbb{N}}$ of $\{x_n\}_{n\in\mathbb{N}}$. Since for any sub sequence $\{x_{n_{k_m}}\}_{m\in\mathbb{N}}$ of $\{x_{n_k}\}_{k\in\mathbb{N}}$ , we can see that for $0<\varepsilon<1$  we have $|x_{n_{k_m}} - 1|> 1 +\varepsilon$ for all $m$. Hence for this choice of $\varepsilon$, there does not exists any $N(\varepsilon)\in \mathbb{N}$ for which $|x_{n_{k_m}} - 1|< 1 +\varepsilon $ holds for all $m\geq N(\varepsilon)$. Therefore there does not exists any $M'=\{m'_1<m'_2<\cdots< m'_k<\cdots\}\in F(I)$ for which $\{x_{n_{m'_k}}\}_{k\in \mathbb{N}}$ is rough convergent to $1$ of roughness degree $r=1$. So $1$ is not a rough $I^*$-limit of roughness degree $r=1$ of the sub sequence considered above.
\end{exmp}
\begin{thm}
If $I$ is an ideal which satisfies the condition (AP), then the rough $I^*$-limit set of a sequence $\{x_n\}_{n\in\mathbb{N}}$ of some roughness degree $r$ is a subset of rough $I^*$-limit set of a sub sequence $\{x_{n_k}\}_{k\in\mathbb{N}}$ of same roughness degree $r$.
\end{thm}
\begin{proof}
Let $x$ be a rough $I^*$-limit of a sequence $\{x_n\}_{n\in\mathbb{N}}$ of some roughness degree $r$. As a rough $I^*$-limit is also a rough $I$-limit of $\{x_n\}_{n\in\mathbb{N}}$, so $x$ is a rough $I$-limit of $\{x_n\}_{n\in\mathbb{N}}$. Since rough $I$-limit of a sequence $\{x_n\}_{n\in\mathbb{N}}$ is a subset of rough $I$-limit of a sub sequence $\{x_{n_k}\}_{k\in\mathbb{N}}$, hence $x$ is rough $I$-limit of the sub sequence $\{x_{n_k}\}_{k\in\mathbb{N}}$. Now as $I$ satisfies the condition (AP), so $x$ is also a rough $I^*$-limit of the sub sequence $\{x_{n_k}\}_{k\in\mathbb{N}}$. 
\end{proof}
\begin{thm}\label{th3}
Let $I$ and $K$ be two admissible ideals on $\mathbb{N}$ and $r$ be a non-negative real number. Suppose that a sequence $\{x_n\}_{n\in\mathbb{N}}$ in a normed linear space $(X, ||\cdot||)$ is rough $I^K$-convergent to $x$ of roughness degree $r$ then $\{x_n\}_{n\in\mathbb{N}}$ is also rough $I$-convergent to $x$ of same roughness degree $r$ if $K\subset I$.
\end{thm}
\begin{proof}
Let $I$ and $K$ be two admissible ideals on $\mathbb{N}$ such that $K\subset I$. Also let $r$ be a non-negative real number. Suppose that a sequence $\{x_n\}_{n\in\mathbb{N}}$ is rough $I^K$-convergent to $x$ of roughness degree $r$. Then there exists a set $M=\{m_1<m_2<\cdots<m_k<\cdots\}\in F(I)$ such that for any $\varepsilon >0$ the set $A(\varepsilon)=\{k\in \mathbb{N} : ||x_{m_k} - x || \geq r +\varepsilon\}\in K|M$. Suppose that $A(\varepsilon)=\{k\in\mathbb{N}: || x_{m_k} - x||\geq r +\varepsilon\}=K_1 \cap M$ for some $K_1\in K$. Now as $K$ is an ideal and $K_1\cap M\subset K_1$, so $K_1\cap M\in K$. Again $\{n\in \mathbb{N} : ||x_n - x|| \geq r +\varepsilon\}\subset (K_1\cap M)\cup \mathbb{N}\setminus M $. Since $\mathbb{N}\setminus M \in I$ and $K\subset I$, therefore $(K_1\cap M)\cup \mathbb{N}\setminus M\in I$. So $\{n\in \mathbb{N} : ||x_n - x|| \geq r +\varepsilon\}\in I$. Hence the result follows.
\end{proof}
Converse part of the theorem \ref{th3} is also valid, i.e.,  if a rough $I^K$-limit $x$ of a sequence $\{x_n\}_{n\in\mathbb{N}}$ of some roughness degree $r$ implies that $x$ is also a rough $I$-limit of same roughness degree $r$ then $K\subset I$. To prove this we need the following lemma.
\begin{lem}\label{lemma 3.5}
If $I$ and $K$ are ideals on $\mathbb{N}$. Then a rough $K$-limit of a sequence $\{x_n\}_{n\in\mathbb{N}}$ of some roughness degree $r$ is also a rough $I^K$-limit of $\{x_n\}_{n\in\mathbb{N}}$ of same roughness degree $r$.
\end{lem}
\begin{proof}
Let $I$ and $K$ be two ideals on $\mathbb{N}$ and $r$ be a non-negative real number. Let $x$ be a rough $K$-limit  of $\{x_n\}_{n\in\mathbb{N}}$ of roughness degree $r$ i.e., $x\in K-LIM^r x_n$. Then for any $\varepsilon >0$, $\{n\in\mathbb{N}: ||x_n - x||\geq r + \varepsilon\}\in K$. Now as $\phi \in I$, so $\mathbb{N}\in F(I)$. Let $M=\{m_1<m_2<\cdots<m_k<\cdots\}=\mathbb{N}\in F(I)$, then $\{x_{m_k}\}=\{x_n\}$ and $K|M = K$. Therefore $\{k\in \mathbb{N}: ||x_{m_k} - x||\geq r +\varepsilon \}= \{n \in\mathbb{N} : ||x_n - x||\geq r +\varepsilon\}\in K=K|M$. So $x\in I^K-LIM^r x_n$. Hence the result follows.
\end{proof}
\begin{thm}\label{corproof}
If rough $I^K$-limit of a sequence $\{x_n\}_{n\in\mathbb{N}}$ of some roughness degree $r$ is $x$ implies that $x$ is also a rough $I$-limit of $\{x_n\}_{n\in\mathbb{N}}$ of same roughness degree $r$ then $K\subset I $.
\end{thm}
\begin{proof}
Suppose that $K\not\subset I$ and $r$ be a non-negative real number. Then there exists a set $A\in K\setminus I$. Let us choose $x, y\in X$ such that $||x||=1$ and $ y = (r + 2 )x$, then we have $|| x - y||\geq r + \varepsilon$ for $0< \varepsilon \leq 1$ and  $||x - y|| < r +\varepsilon$ for $\varepsilon > 1$. Now define a sequence $\{x_n\}_{n\in\mathbb{N}}$ as follows $x_n=
\begin{cases}
x + r , & n\in \mathbb{N}\setminus A\\
y, & n\in A
\end{cases}$. Then for any $\varepsilon >0$ the set $\{n\in\mathbb{N}: || x_n - x||\geq r +\varepsilon\}$ is either the set $A$ (when $0<\varepsilon\leq 1$) or $\phi$ (when $\varepsilon>1$). Since $K$ is an admissible ideal and $A\in K$, therefore $\{n\in\mathbb{N}: || x_n - x||\geq r +\varepsilon\}\in K$. Thus $x$ is a rough $K$-limit of $\{x_n\}_{n\in\mathbb{N}}$ of roughness degree $r$. Now by lemma \ref{lemma 3.5}, $x$ is a rough $I^K$-limit of $\{x_n\}_{n\in\mathbb{N}}$ of roughness degree $r$. Since for $0<\varepsilon \leq 1$, $\{n\in\mathbb{N}: ||x_n - x||\geq r +\varepsilon\}= A$ and $A\notin I$. So $\{n\in\mathbb{N}: ||x_n - x||\geq r +\varepsilon\}\notin I$ and hence $x$ is not a rough $I$-limit of roughness degree $r$. But by our assumption, $x$ is also a rough $I$-limit of $\{x_n\}_{n\in\mathbb{N}}$. Thus we arrive at a contradiction and so, $K\subset I$.  
\end{proof}
\begin{cor}
Let $I$ and $K$ be two ideals on $\mathbb{N}$. Then rough $I^K$-limit set is a subset of $I$-limit set of a sequence $\{x_n\}_{n\in\mathbb{N}}$ of some roughness degree $r$ if and only if $K\subset I$.
\end{cor}
\begin{proof}
In view of Theorem \ref{th3} and Theorem \ref{corproof} the result follows.
\end{proof}
In general, if $x$ is a rough $I$-limit of a sequence $\{x_n\}_{n\in\mathbb{N}}$ in a normed linear space then it does not necessarily implies that $x$ is also a rough $I^K$-limit of $\{x_n\}_{n\in\mathbb{N}}$. The following is an example in support of this assertion.
\begin{exmp}
Let $I$ be ideal as in example \ref{3.2}. Also let $K$ be the ideal on $\mathbb{N}$ such that it is the collection of all subsets of $\mathbb{N}$ whose natural density is zero. Now let us define a sequence in real numbers space with usual norm by $x_n=\frac{1}{j}$ if $n\in D_j$. Let $r$ be a arbitrary non-negative real number. Let $\varepsilon >0$ be arbitrarily chosen, then there exists a $l\in \mathbb{N}$ such that $\varepsilon >\frac{1}{l}$. Clearly $[-r, r]\subset I-LIM^r x_n$, as $\{n\in\mathbb{N} : |x_n - x |\geq r + \varepsilon\}\subset  D_1 \cup D_2 \cup \cdots D_l\in I$ for any $x\in [-r, r]$.\\
If possible let $-r$ be a rough $I^K$-limit of $\{x_n\}$ of roughness degree $r$. Then there exists a $M=\{m_1<m_2<\cdots<m_k<\cdots\}\in F(I)$ such that the sub sequence $\{x_{m_k}\}$ is rough $K|M$-convergent of roughness degree $r$ to $-r$. Now as $N\setminus M=H(\text{say})\in I$, so there exists a $p\in \mathbb{N}$ such that $H\subset D_1\cup D_2 \cup\cdots \cup D_p$. Hence $D_k\subset M$ for all $k\geq p +1$. Let $\varepsilon =\frac{1}{p+1}$. Then $\{k\in\mathbb{N}: ||x_{m_k} + r||\geq r + \frac{1}{p+1}\}=\{k\in\mathbb{N}: m_k\in D_{p+1}\}$. As $D_{p+1}=\{2^{p}(2s -1): s=1,2,\cdots\}$ and natural density of $D_{p+1}=\frac{1}{2^{p+1}}$, therefore natural density of the set $\{k\in\mathbb{N}: ||x_{m_k} + r||\geq r + \frac{1}{p+1}\}$ is not zero. Hence $\{k\in\mathbb{N}: ||x_{m_k} + r||\geq r + \frac{1}{p+1}\}\notin K|M$, since natural density of each set belongs to $K|M$ is also zero. Therefore $-r$ is not a rough $I^K$-limit of $\{x_n\}$ of roughness degree $r$.

\end{exmp}
Rough $I$-limit $x$ of a sequence $\{x_n\}_{n\in\mathbb{N}}$ is also a rough $I^K$-limit of $\{x_n\}_{n\in\mathbb{N}}$ if the ideal $I$ satisfies the condition (AP). Thus we have the following Theorem.
\begin{thm}\label{1}
Let $I$ and $K$ be two admissible ideals on $\mathbb{N}$ such that the ideal $I$ satisfies the condition (AP). Also let $r$ be a non-negative real number and $\{x_n\}_{n\in\mathbb{N}}$ be a sequence in a normed linear space $(X, ||\cdot||)$. Then $x\in I-LIM^r x_n$ implies $x\in I^K-LIM^r x_n$.
\end{thm}
\begin{proof}
Suppose that $I$ and $K$ be two ideals on $\mathbb{N}$ such that the ideal $I$ satisfies the condition (AP). Let $\{x_n\}_{n\in\mathbb{N}}$ be sequence such that $x\in I-LIM^r x_n$. Now since $I$ satisfies the condition (AP), hence $x\in I^*-LIM^r x_n$. Now as $I^*-LIM^r x_n \subset I^K-LIM^r x_n$, therefore $x\in I^K-LIM^r x_n$.
\end{proof}

\begin{thm}\label{2}
Let $I$ and $K$ be two ideals on $\mathbb{N}$. If for any sequence $\{x_n\}_{n\in\mathbb{N}}$  in a normed linear space $(X, ||\cdot||)$ the implication $ I- LIM^r x_n \subset I^K- LIM^r x_n$ holds, then the the ideal $I$ has the additive property with respect to $K$, i.e., AP($I$, $K$) holds.
\end{thm}
\begin{proof}
Let $\{A_n\}_{n\in \mathbb{N}}$ be a sequence of mutually disjoint sets  belonging to $I$ and $r>0$ be a  real number. Define a sequence $\{x_n\}_{n\in\mathbb{N}}$ in real number space with usual norm as follows $x_n=\begin{cases}
r+\frac{1}{i},\:\:\:\:\: n\in A_i\\
0,\: \:\:\:\:\:\:  n\in \mathbb{N}\setminus \displaystyle{\cup_i} A_i
\end{cases}$. Then for any $\varepsilon>0$, there exists $p\in\mathbb{N}$ such that $\{n\in\mathbb{N}: || x_n - 0|| \geq r +\varepsilon\}\subset A_1\cup A_2\cup\cdots \cup A_p\in I$. So $0\in I-LIM^r x_n$. Consequently, by our assumption, $0\in I^K-LIM^r x_n$. Thus there exists a set $M=\{m_1<m_2<\cdots< m_k<\cdots\}\in F(I)$ such that the sub sequence $\{x_{m_k}\}_{k\in \mathbb{N}}$ is rough $K|M$-convergent to $0$ of roughness degree $r$. Now if $\displaystyle{\cup_i} A_i\in I$ then by taking $A_i =B_i$  for $i\in\mathbb{N}$ the results follows directly by using $(iv)$ of the lemma \ref{apk}. So let $\displaystyle{\cup_i} A_i\notin I$. Since $M\in F(I)$, so the set $M$ contains a infinite numbers of $A_i$'s. Now for  arbitrary $\varepsilon>0$ the set $\{k\in \mathbb{N}: |x_{m_k} - 0|\geq r+\varepsilon\}\in K|M$. For each $i\in\mathbb{N}$ either $A_i\cap M\neq \phi$ or $A_i\cap M=\phi$. By the construction of the sequence and by the fact that $0\in I^K-LIM^r x_n$ in both cases we have $A_i\cap M\in K|M$ and since $\varepsilon>0$ is arbitrary so, $A_i\cap M\in K$ for each $i\in\mathbb{N}$. Now as $M\in F(I)$, so $\mathbb{N}\setminus M=B(\text{say})\in I$. Let us put $B_i =A_i \cap B$ for each $i$. Then each of $B_i$ belongs to $I$. Also as $\displaystyle{\bigcup_{i=1}^{\infty}}B_i= \displaystyle{\bigcup_{i=1}^{\infty}}(A_i \cap B)= B \cap \displaystyle{\bigcup_{i=1}^{\infty}} A_i \subset B$, so $\displaystyle{\bigcup_{i=1}^{\infty}}B_i\in I$. Now as $B_i \subset A_i$, so $A_i\setminus B_i = A_i\cap M$. Thus $A_i\setminus B_i\in K$. Therefore $A_i \sim_K B_i$ for $i\in \mathbb{N}$, since $A_i \sim_K B_i \Leftrightarrow A_i \Delta B_i\in K$ and $A_i \Delta B_i = A_i \setminus B_i$ in this case. Thus by the virtue $(iv)$ of the lemma \ref{apk} the result follows.
\end{proof}

In the next example we will see that, as in the case rough $I^*$-limit, rough $I^K$-limit set of a sequence $\{x_n\}_{n\in\mathbb{N}}$ in a normed linear space may not be a subset of rough $I^K$-limit set of a sub sequence $\{x_{n_k}\}_{k\in\mathbb{N}}$. 
\begin{exmp}
Let $I$ be the collection of all subsets of $\mathbb{N}$ whose natural density is zero. Then $I$ is a non trivial admissible ideal on $\mathbb{N}$. Also let $\mathbb{N}= \displaystyle{\bigcup_{j=1} ^ \infty} D_j$ be a decomposition of $\mathbb{N}$ such that $D_j=\{2^{j-1}(2s -1): s=1, 2, \cdots\}$. Then each $D_j$ is infinite and $D_j \cap D_k =\phi$ for $j\neq k$. Now $K$ be the ideal such that it is the class of all subsets of $\mathbb{N}$ which intersects with only a finite numbers of $D_j$'s. Let us consider the sequence in real number space with usual norm, where $x_n =\frac{1}{j}$ if $n\in D_j$. Now as $\phi\in I$, so $\mathbb{N}\in F(I)$. Let us take $\mathbb{N}=M$. Let us enumerate $M$ as, $M =\{m_1<m_2<m_3<\cdots<m_k<\cdots\}$. Then $K|M= K$. Also let $r >0$ be arbitrary. Now for an arbitrary $\varepsilon >0$ one can find a $l\in \mathbb{N}$ such that $\varepsilon >\frac{1}{l}$. Now we have a $p\in\mathbb{N}$ such that $\{k\in\mathbb{N}: ||x_{m_k} + r||\geq r +\varepsilon\}\subset D_1 \cup D_2\cup \cdots \cup D_p\in K=K|M$. So $\{k\in\mathbb{N}: ||x_{m_k} + r||\geq r +\varepsilon\}\in K|M$. Therefore $-r$ is a rough $I^K$-limit of roughness degree $r$. Again let us consider the sub sequence $\{x_{n_k}\}$ of the sequence $\{x_n\}$ such that $\{x_{n_k}\}=1$ for all $k$'s. Then for any $0<\varepsilon<1$ and for any $M=\{m_1<m_2<\cdot<m_k<\cdots\}\in F(I)$, $\{k\in\mathbb{N}: ||x_{n_{m_k}} + r||\geq r +\varepsilon\}=\mathbb{N}$. Since $\mathbb{N}\notin K|M$, therefore $\{k\in\mathbb{N}: ||x_{m_k} + r||\geq r +\varepsilon\}\notin K|M$. Hence $-r$ is not a rough $I^K$-limit of the sub sequence $\{x_{m_k}\}$.
\end{exmp}
\begin{thm}\label{newone}
Let $I$ and $K$ be two admissible ideal on $\mathbb{N}$ such that $K\subset I$ and AP($I$, $K$) holds. Then the rough $I^K$-limit set of a sequence $\{x_n\}_{n\in\mathbb{N}}$ of some roughness degree $r$ is a subset of rough $I^K$-limit set of a sub sequence $\{x_{n_k}\}_{k\in\mathbb{N}}$ of same roughness degree $r$.
\end{thm}
\begin{proof}
Let $x$ be a rough $I^K$-limit of a sequence $\{x_n\}_{n\in\mathbb{N}}$ of some roughness degree $r$. Also let $I$ and $K$ be two ideals on $\mathbb{N}$ such that AP($I$, $K$) holds and $K\subset I$. Now since $K\subset I$, therefore $x$ is also a rough $I$-limit of the sequence $\{x_n\}_{n\in\mathbb{N}}$. So $x$ is also a rough $I$-limit of a sub sequence $\{x_{n_k}\}_{k\in\mathbb{N}}$ of the sequence $\{x_n\}_{n\in\mathbb{N}}$. Again since AP($I$, $K$) holds, so $x$ is a rough $I^K$-limit of the sub sequence $\{x_{n_k}\}_{k\in\mathbb{N}}$. Hence the results follows.
\end{proof}
\begin{rem}
Rough $I^K$-limit of a sequence $\{x_n\}_{n\in\mathbb{N}}$ of some roughness degree $r$ is also a rough $I*$-limit of $\{x_n\}_{n\in\mathbb{N}}$ of same roughness degree $r$ if some additional condition holds as given in the following Theorem.
\end{rem}
\begin{thm}
Let $I$ and $K$ be two admissible ideal on $\mathbb{N}$ such that $K\subset I$ and the ideal $I$ satisfies the condition (AP). Then for a sequence $\{x_n\}_{n\in\mathbb{N}}$ in $X$ rough $I^K$-limit of some roughness degree $r$ is a rough $I^*$-limit of same roughness degree $r$.
\end{thm}
\begin{proof}
Let $x$ be a rough $I^K$-limit of a sequence $\{x_n\}_{n\in\mathbb{N}}$ of some roughness degree $r$. Also let $I$ and $K$ be two ideals on $\mathbb{N}$ such that AP($I$, $K$) holds and $K\subset I$. Again since $K\subset I$, so $x$ is also a rough $I$-limit of the sequence $\{x_n\}_{n\in\mathbb{N}}$. Again since the ideal $I$ satisfies the condition (AP), so $x$ is a rough $I^*$-limit of the sequence $\{x_n\}_{n\in\mathbb{N}}$ by Theorem \ref{th2}.
\end{proof}
\begin{defn} [c.f. \cite{R7}]
Let $I$ and $K$ be two admissible ideals on $\mathbb{N}$. A sequence $\{x_n\}_{n\in\mathbb{N}}$ in a normed linear space is said to be $K|M$-bounded if there exists a set $M=\{m_1<m_2<\cdots<m_k<\cdots\}\in F(I)$ and a positive number $L$ such that for the sub sequence $\{x_{m_k}\}_{k\in \mathbb{N}}$ we have $\{k\in \mathbb{N}: || x_{m_k}||\geq L\}\in K|M$.
\end{defn}
Obviously for a bounded sequence $\{x_n\}_{n\in\mathbb{N}}$ there always exists a non-negative real number $r$ for which $I^K-LIM^r x_n\neq \phi$. The reverse implication is generally not valid which can be seen from the following example, but if we take the sequence $\{x_n\}_{n\in\mathbb{N}}$ to be $K|M$-bounded then the reserve implication also holds. 
\begin{exmp}\label{newexmp}
Let $I$ be the ideal of the class of all those subsets of $\mathbb{N}$ whose natural density is zero. Then $I$ is an admissible ideal on $\mathbb{N}$. Also let $K$ be any admissible ideal on $\mathbb{N}$. Let us consider the sequence $\{x_n\}_{n\in\mathbb{N}}$ in real number space with usual norm as follows: $x_n=\begin{cases}
1, \:\: n\neq k^2\\
n,\:\: n=k^2
\end{cases}$, for some $k\in\mathbb{N}$. Now as $A(\text{say}) =\{n\in\mathbb{N}: n=k^2 \:\text{for some}\: k\in\mathbb{N}\}\in I$, so $\mathbb{N}\setminus A=M(\text{say})\in F(I)$. Let us enumerate $M$ as, $M=\{m_1<m_2<\cdots< m_k<\cdots\}$. Now we see that for $r=1$ and for any $\varepsilon>0$, the set $\{k\in\mathbb{N}: |x_{m_k} - x|\geq r + \varepsilon\}=\phi\in K|M$ for any $x\in [0, 2]$. So $[0, 2]\subset I^K-LIM^r x_n$. But the sequence considered here is unbounded.
\end{exmp}
\begin{thm}
Let $I$ and $K$ be two admissible ideals on $\mathbb{N}$. Then a sequence $\{x_n\}_{n\in\mathbb{N}}$ is $K|M$-bounded if and only if there exists a non-negative real number $r$ such that $I^K-LIM^ r x_n\neq \phi$.  
\end{thm}
\begin{proof}
Suppose that the sequence $\{x_n\}_{n\in\mathbb{N}}$ is $K|M$-bounded. Then there exists a set $M=\{m_1<m_2<\cdots <m_k<\cdots\}\in F(I)$  and a positive number $L$ such that for the sub sequence $\{x_{m_k}\}$, $\{k\in\mathbb{N}: || x_{m_k}||\geq L \}=K(\text{say})\in K|M$. Define $r:=\sup{\{||x_{m_k}||: k\in K^\complement\}}$, where $K^\complement$ denote the complement of $K$ in $\mathbb{N}$. Then for any $\varepsilon>0$ we have $\{k\in\mathbb{N}: ||x_{m_k} - 0||\geq r +\varepsilon\}\subset K$. So $0\in I^K-LIM ^r x_n$.\\
Conversely suppose that $I^K-LIM^r x_n\neq \phi$ for some $r\geq 0$. Let $x\in I^K-LIM^r x_n$ and $||x||=L$. Then there exists a set $M=\{m_1<m_2<\cdots<m_k<\cdots\}\in F(I)$ such that for any $\varepsilon>0$ the set $\{k\in\mathbb{N}: ||x_{m_k} - x||\geq r +\varepsilon\}= K_1(\text{say})\in K|M$. Now $||x_{m_k}||=||x_{m_k} - x +x||\leq ||x_{m_k} - x|| + ||x||< r +\varepsilon + L$ for all $k\in K_1^\complement$. So $\{k\in\mathbb{N}: ||x_{m_k}||\geq r +\varepsilon + L\}\in K|M$. So the sequence $\{x_n\}_{n\in\mathbb{N}}$ is $K|M$-bounded.
\end{proof}
It is remarkable that the rough $I^K$-limit set of a sequence $\{x_n\}_{n\in\mathbb{N}}$ is convex set, as shown in the following theorem.
\begin{thm}
Let $I$ and $K$ be two ideals on $\mathbb{N}$ and $r$ be a non-negative real number. Then for a sequence $\{x_n\}_{n\in\mathbb{N}}$, the rough $I^K$-limit set $I^K-LIM^r x_n$ is convex.
\end{thm} 
\begin{proof}
Let us assume that $x_1, x_2\in I^K-LIM^r x_n$. Then there exists $M'=\{m'_1<m'_2<\cdots<m'_k<\cdots\}$ and $M''=\{m''_1<m''_2<\cdots<m''_k<\cdots\}$ in $F(I)$ such that $\{k\in\mathbb{N}: ||x_{m'_k} - x_1||\geq r +\varepsilon\}\in K|M'$ and $\{k\in\mathbb{N}: ||x_{m''_k} - x_2||\geq r +\varepsilon\}\in K|M''$. Now as both of $M'$ and $M''$ belongs to $F(I)$. Let $M=M'\cap M''$. Then $M\in F(I)$ and let us enumerate $M$ as, $M=\{m_1<m_2<\cdots<m_k<\cdots\}$. Since $K$ is an admissible ideal, therefore $\{k\in\mathbb{N}: ||x_{m_k} - x_1||\geq r +\varepsilon\}\in K|M\rightarrow(i)$ and also $\{k\in\mathbb{N}: ||x_{m_k} - x_2||\geq r +\varepsilon\}\in K|M\rightarrow(ii)$. Let $0\leq \lambda\leq 1$. Now as $||x_{m_i} - [(1 -\lambda)x_1 + \lambda x_2]||=||(1 -\lambda)(x_{m_i} - x_1) + \lambda (x_{m_i} - x_2)||< r +\varepsilon$ for each $i$ belongs to the complements of the set as in $(i)$ and $(ii)$ simultaneously. So $||k\in\mathbb{N}: ||x_{m_k} - [(1-\lambda)x_1 + \lambda x_2]||\geq r +\varepsilon\}\in K|M$.
\end{proof}
\begin{rem}
Let $I$ be an ideal on $\mathbb{N}$. Since for any non-negative real number $r$ and for a sequence $\{x_n\}_{n\in\mathbb{N}}$ in $X$ rough $I^*$-limit of roughness degree $r$ is also a rough $I^K$-limit of same roughness degree $r$, therefore rough $I^*$-limit set of roughness degree $r$ of $\{x_n\}_{n\in\mathbb{N}}$ is also a convex set.
\end{rem}
\begin{thm}
Let $I$, $I_1$, $I_2$, $K$, $K_1$ and $K_2$ be ideals on $\mathbb{N}$ such that $I_1\subset I_2$ and $K_1\subset K_2$. Also let $\{x_n\}_{n\in\mathbb{N}}$ be a sequence and $r$ be a non-negative real number. Then $(i)$ $I_1 ^ K-LIM^r x_n \subset I_2 ^K-LIM^r x_n$.\\
$(ii)$ $I^{K_1}-LIM^r x_n \subset I^{K_2}-LIM^r x_n$.
\end{thm}
\begin{proof}
$(i)$ Suppose $x\in I_1 ^ K-LIM^r x_n$. Thus there exists a set $M=\{m_1<m_2<\cdots<m_k\cdots\}\in F(I_1)$ such that the sub sequence $\{x_{m_k}\}_{k\in\mathbb{N}}$ is rough $K|M$-convergent to $x$ of roughness degree $r$. Now as $I_1\subset I_2$, therefore $\mathbb{N}\setminus M\in I_1 \subset I_2$. Hence $M\in F(I_2)$. So $x\in I_2 ^K-LIM^r x_n$.\\
$(ii)$ Let $x\in I^{K_1}-LIM^r x_n$. Therefore the exists a set $M=\{m_1<m_2<\cdots< m_k<\cdots\}\in F(I)$ such that the sub sequence $\{x_{m_k}\}_{k\in\mathbb{N}}$ is rough $K_1|M$-convergent to $x$ of roughness degree $r$. Now as $K_1\subset K_2$, therefore the sub sequence $\{x_{m_k}\}_{k\in\mathbb{N}}$ is also rough $K_2|M$-convergent to $x$ of roughness degree $r$. Hence $x\in I^{K_2}-LIM^r x_n$.
\end{proof}
\begin{thm}
Let $I$ and $K$ be two admissible ideals on $\mathbb{N}$. Then rough $I^K$-limit set of a sequence $\{x_n\}_{n\in\mathbb{N}}$ is closed.
\end{thm}
\begin{proof}
The proof is trivial when $I^K-LIM^r x_n=\phi$, where $r\geq 0$. Let us assume that $I^K-LIM^r x_n\neq \phi$ for some $r\geq 0$. Also let $\{y_n\}_{n\in\mathbb{N}}$ be a sequence in $I^K-LIM^r x_n$ such that $y_n \rightarrow y$. Since $y_n\rightarrow y$, so for a given $\varepsilon >0$ there exists a $N_1\in\mathbb{N}$ such that $||y_n - y||<\frac{\varepsilon}{2}$ for all $n>N_1$. Let $n_1> N_1$, therefore $||y_{n_1} - y|| < \frac{\varepsilon}{2}$. Again since $\{y_n\}_{n\in\mathbb{N}}$ is a sequence in $I^K-LIM^r x_n$, therefore $y_{n_1}\in I^K-LIM^r x_n$. Therefore there exists a set $M=\{m_1<m_2<\cdots<m_k<\cdots\}\in F(I)$ such that $\{k\in\mathbb{N}: ||x_{m_k} - y_{n_1}||\geq r +\frac{\varepsilon}{2}\}=K_1(\text{say})\in K|M$. Now for $k\in K_1^\complement$ we have, $|| x_{m_k} - y||= || x_{m_k} - y_{n_1} + y_{n_1} - y||\leq ||x_{m_k} - y_{n_1}|| +|| y_{n_1} - y||< r +\frac{\varepsilon}{2} +\frac{\varepsilon}{2}= r+\varepsilon$. Thus $\{k\in\mathbb{N}: ||x_{m_k} - y||\geq  r +\varepsilon\}\in K|M$. Therefore $y\in I^K-LIM^ r x_n$ and hence the result follows.
\end{proof}
 
\end{document}